\documentclass[conference]{IEEEtran}
\IEEEoverridecommandlockouts
\usepackage{cite}
\usepackage{amsmath,amssymb,amsfonts,url,bm,enumerate}
\usepackage{amsthm}
\usepackage{algorithmic}
\usepackage{graphicx}
\usepackage{textcomp}
\usepackage{xcolor}
\usepackage{tikz}
\usepackage{float}
\usepackage{makecell}
\usepackage{ctable}
\usepackage{booktabs}
\usepackage{array}
\usepackage{multicol}
\def\BibTeX{{\rm B\kern-.05em{\sc i\kern-.025em b}\kern-.08em
    T\kern-.1667em\lower.7ex\hbox{E}\kern-.125emX}}
    
\usepackage{hyperref}

\theoremstyle{plain}
\newtheorem{theorem}{Theorem}
\newtheorem{lemma}[theorem]{Lemma}
\newtheorem{proposition}[theorem]{Proposition}

\newtheorem{corollary}[theorem]{Corollary}

\theoremstyle{definition}
\newtheorem{definition}[theorem]{Definition}

\newtheorem{remark}[theorem]{Remark}

\newtheorem{example}[theorem]{Example}

\newcommand{\cc}{^{\perp\perp}}
\renewcommand{\c}{^\perp}

\newcommand{\notperp}{\mathbin{\not\perp}}

\definecolor{darkgreen}{rgb}{0.0, 0.5, 0.0}

\newcommand{\lin}[1]{[#1]}
\newcommand{\herm}[2]{\left( #1 , #2 \right)}

\definecolor{brightgreen}{rgb}{0.4, 1.0, 0.0}
\newcolumntype{?}{!{\vrule width 1pt}}

\usepackage[most]{tcolorbox}

\def \rev {\textcolor{black}}

\begin{document}

\title{Linear orthogonality spaces as a \\ new approach to quantum logic\\
}
\author{
\IEEEauthorblockN{Kadir~Emir}
\IEEEauthorblockA{\textit{Department of Mathematics and Statistics} \\
\textit{Faculty of Science, Masaryk University}\\
Brno, Czech Republic\\
emir@math.muni.cz}
\and
\IEEEauthorblockN{David~Kruml}
\IEEEauthorblockA{\textit{Department of Mathematics and Statistics} \\
\textit{Faculty of Science, Masaryk University}\\
Brno, Czech Republic\\
kruml@math.muni.cz}
\and
\IEEEauthorblockN{Jan~Paseka}
\IEEEauthorblockA{\textit{Department of Mathematics and Statistics} \\
\textit{Faculty of Science, Masaryk University}\\
Brno, Czech Republic\\
paseka@math.muni.cz}
\and
\IEEEauthorblockN{Thomas~Vetterlein}
\IEEEauthorblockA{\textit{Department of Knowledge-Based Mathematical Systems}\\
\textit{Johannes Kepler University Linz}\\
Linz, Austria\\
Thomas.Vetterlein@jku.at}
}

\maketitle

\begin{abstract}
The notion of an orthogonality space was recently rediscovered as an effective means to characterise the essential properties of quantum logic. The approach can be considered as minimalistic; solely the aspect of mutual exclusiveness is taken into account. In fact, an orthogonality space is simply a set endowed with a symmetric and irreflexive binary relation. If the rank is at least $4$ and if a certain combinatorial condition holds, these relational structures can be shown to give rise in a unique way to Hermitian spaces. In this paper, we focus on the finite case. In particular, we investigate orthogonality spaces of rank at most $3$.
\end{abstract}

\begin{IEEEkeywords}
Orthogonality spaces, undirected graphs, linear orthogonality spaces, finite rank
\end{IEEEkeywords}

\section{Introduction}

To grasp the essential properties of the basic model used in quantum physics, David Foulis and his collaborators coined in the 1970s the notion of an orthogonality space \cite{Dac,Wlc}. The idea was to reduce the involved structure of a complex Hilbert space to the minimum of what is really needed. An orthogonality space is a set endowed with a binary relation about which not more than symmetry and irreflexivity is assumed. The canonical example is the projective Hilbert space together with the orthogonality relation. The approach can be seen as an attempt to increase the level of abstraction in quantum logic to its limits: From the quantum-physical perspective, solely the aspect of distinguishability of measurement results is taken into account; from the logical perspective, solely the aspect of mutual exclusiveness is exploited.

Orthogonality spaces were recently rediscovered and they have proven as a basis of quantum logic in an amazingly effective way \cite{nos,Vet1,Vet2,Vet3}. In fact, each orthogonality space gives rise to a test space; see \cite{Wlc}. Test spaces can in turn be understood as an abstract way to model quantum-mechanical propositions. The latter are, in the standard approach, modelled by subspaces of Hilbert spaces. It has turned out that the transition from orthogonality spaces to inner-product spaces is possible on the basis of a remarkably simple condition, to which we refer to as linearity. The rank of an orthogonality space is, loosely speaking, the maximal number of mutually orthogonal elements. In case that the rank is at least $4$, linearity is sufficient to lead from the simple relational structure of an orthogonality space to a Hermitian space.

Accordingly, the focus of investigations of orthogonality spaces has up to now mostly been on the case that the rank is $4$ or higher. In contrast, this paper is focussed on the case that the orthogonality space is finite and of rank $2$ or $3$, where there is no general representation theory. Often adopting the point of view of graph theory, we establish for such spaces a number of interesting combinatorial properties.

\section{Basic notions, definitions and results}

First, we recall some basic concepts. We also refer to \cite{OML} and \cite{GT} for the notions concerning \rev{modular} and orthomodular structures, and graph theory.

\begin{definition}
An {\em orthogonality space} is a non-empty set $X$ equipped with a symmetric, irreflexive binary relation $\perp$, 
called the {\em orthogonality relation}. The supremum of the cardinalities of sets of mutually orthogonal elements of $X$ is called the {\em rank} of~$(X, \perp)$.
\end{definition}

Recall that orthogonality spaces are essentially the same as undirected graphs, understood such that the edges are two-element subsets of the set of nodes (in graph theory called vertices). The rank of an orthogonality space under this identification is the supremum of the sizes of cliques \rev{-- a clique is a subset of vertices of an undirected graph such that every two distinct vertices are adjacent}. 

Consequently, we have the following analogy:
	
		\begin{center}
			\begin{tabular}{c c c} 
				Undirected graphs & $\Leftrightarrow$ & Orthogonality spaces \\
				Adjacents & $\Leftrightarrow$ & Orthogonal elements \\
				Cliques & $\Leftrightarrow$ & Orthogonal subsets \\
				Maximal cliques & $\Leftrightarrow$ & Maximal orthogonal subsets
			\end{tabular} 
		\end{center}  
		
In the sequel, we will use both kinds of notations interchangeably.		

We sometimes prefer to specify an orthogonality space by means of
\begin{align*}
	M_{\perp} \, := \, \textrm{ set of its maximal orthogonal subsets}.
	\end{align*}

\begin{definition}
	A  {\em path} in an orthogonality space is a sequence of distinct vertices  such that adjacent vertices in the
	sequence are orthogonal in the orthogonality space. The {\em length} of a path is the number of edges on the path. 
	The {\em distance} between two vertices $a$ and $b$, denoted {by} $d(a,b)$, is the length of a shortest $a - b$ path 
	if any; otherwise $d(a,b)=\infty$. We say that an orthogonality space is {\em connected} if it is connected in the graph theoretical sense.
	The {\em diameter} {$d(X)$} of a connected non-trivial orthogonality space $(X,\perp)$ is the supremum of the distances between any pair of different vertices. 
\end{definition} 

Note that the distance function is a metric on the vertex set of an orthogonality space, 

\begin{enumerate}[{\rm(D1)}]
	\item $a\perp b$  if and only if $d(a,b)=1$, and 
	\item[{\rm(D2)}]  if $A$ is a subset of $X$, ${\perp_A}={\perp} \cap (A\times A)$, $(A, \perp_A)$ is connected and $|A|\geq 2$, 
then $(A, \perp_A)$ has diameter 1 if and only if $A$ is an orthogonal subset of $(X,\perp)$. 
\end{enumerate}

But we are not motivated by graph theory, our primary example  originates in quantum physics.

\begin{example} \label{ex:standard-example-1}
Let $H$ be a Hilbert space. Then the set $P(H)$ of one-dimensional subspaces of $H$, together with the usual orthogonality relation, is an orthogonality space, whose rank coincides with the dimension of $H$.
\end{example}

For an orthogonality space $(X,\perp)$, the orthogonal complement of $A \subseteq X$ is given by:
\begin{align*}
 A ^{\perp} \, = \, \{ x \in X \colon x \perp a, \, \text{for all}\ a \in A \}. 
 \end{align*}

The \rev{unary operation on the power set $\mathcal{P}(X)$ that sends} $A$ to $A\cc$ is a closure operator on $X$. We call the closed subsets {\it orthoclosed} and we denote the collection of orthoclosed subsets by ${\mathcal C}(X, \perp)$.

\begin{definition}
	An orthogonality space $(X, \perp)$ is called {\it linear} if, for any two distinct elements $e, f \in X$, there is a third element $g$ such that $\{e,f\}\c = \{e,g\}\c$ and exactly one of $f$ and $g$ is orthogonal to $e$.
\end{definition}

In other words, for $(X, \perp)$ to be linear means that, for any two distinct elements $e,f \in X$:
\begin{itemize}
	\item[(L1)] if $e \notperp f$, there exists a $g \perp e$ such that $\{ e, f \}\c = \{ e, g \}\c$
	
	\item[(L2)] if $e \perp f$, there exists a $g \notperp e$ such that $\{ e, f \}\c = \{ e, g \}\c$.
\end{itemize}
Note that in both cases $g$ is necessarily distinct from $e$ and $f$. This immediately implies that 
any non-trivial linear  orthogonality space $(X, \perp)$ has at least 3 elements. 

Evidently, if an orthogonality space $(X, \perp)$ fulfills (L1) and $A\in {\mathcal C}(X, \perp)$ then the 
orthogonality space $(A, \perp_A)$ with the induced orthogonality relation $\perp_A \, = \, \perp\cap \,\, (A\times A)$ fulfills (L1) as well.


\begin{proposition}
	There is no implication between the conditions (L1) and (L2) in the linear orthogonality space definition.
\end{proposition}
\begin{proof}
	Let $X$ be a 6 element set. Then the orthogonality space $(X,\perp)$ where
	\begin{align*}
		M_\perp \, = \{\{0, 1, 2, 3\}, \{0, 1, 4, 5\}\}
	\end{align*}
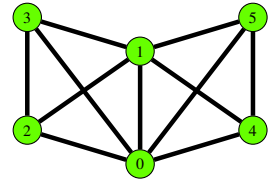
	\begin{align*}
	\begin{tikzpicture}[scale=1.5,every node/.style={draw=black,scale=0.6,circle,fill=brightgreen}]
	\node (0) at (0,0.3) {2};
	\node (2) at (0,1.3) {3};
	\node (4) at (1,0) {0};
	\node (5) at (1,1) {1};
	\node (1) at (2,0.3) {4};
	\node (3) at (2,1.3) {5};
	\draw[line width=0.6mm] (0) to (2);
	\draw[line width=0.6mm] (0) to (4);
	\draw[line width=0.6mm] (0) to (5);
	\draw[line width=0.6mm] (2) to (5);
	\draw[line width=0.6mm] (4) to (5);
	\draw[line width=0.6mm] (1) to (4);
	\draw[line width=0.6mm] (2) to (4);
	\draw[line width=0.6mm] (1) to (3);
	\draw[line width=0.6mm] (1) to (5);
	\draw[line width=0.6mm] (3) to (4);
	\draw[line width=0.6mm] (3) to (5);
	\end{tikzpicture}
	\end{align*}
	fulfills (L1). However, it does not fulfill (L2) since $2 \perp 0$ and $\{2,0\}^{\perp}= \{3,1\}$ but there is no $g \notperp 2$ such that \mbox{$\{2,g\}^{\perp} = \{3,1\}$}. 
	
	On the other hand, for a 7 element set $Y$, the orthogonality space $(Y,\perp')$ where
	\begin{align*}
		M_{\perp'} \, = \{\{0, 4\}, \{0, 6\}, \{5, 1\}, \{5, 2\}, \{5, 3\}, \{6, 1\}, \{6, 2\}\}
	\end{align*}
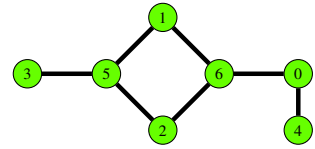
	\begin{align*}
	\begin{tikzpicture}[scale=1.5,every node/.style={draw=black,scale=0.6,circle,fill=brightgreen}]
	\node (3) at (0.3,0) {3};
	\node (5) at (1,0) {5};
	\node (1) at (1.5,0.5) {1};
	\node (2) at (1.5,-0.5) {2};
	\node (6) at (2,0) {6};
	\node (0) at (2.7,0) {0};
	\node (4) at (2.7,-0.5) {4};
	\draw[line width=0.6mm] (3) to (5);
	\draw[line width=0.6mm] (5) to (1);
	\draw[line width=0.6mm] (5) to (2);
	\draw[line width=0.6mm] (2) to (6);
	\draw[line width=0.6mm] (6) to (1);
	\draw[line width=0.6mm] (6) to (0);
	\draw[line width=0.6mm] (0) to (4);
	\end{tikzpicture}
	\end{align*}
	fulfills (L2). However, it does not fulfill (L1) since $0 \notperp 2$ and $\{0,2\}^{\perp} = \{6\}$ but there is no $g \perp 0$ such that \mbox{$\{0,g\}^{\perp}=\{6\}$}. 
\end{proof}

The number of orthogonality spaces (up to isomorphism) fulfilling (L1) is given in \rev{Table \ref{table-l1}}.

\begin{table}[H] 
	\begin{center}
		\caption{}
		\label{table-l1}
		\begin{tabular}{? c | c c | c c ?} 
			\specialrule{.1em}{.1em}{0em} 
			& \multicolumn{2}{c|}{Numbers of all} & \multicolumn{2}{c?}{Numbers of connected} \\ 
			$|X|$ & OS &  (L1)-OS & OS &  (L1)-OS  \\
			\hline
			2 & 2 & 1 & 1 & 1 \\
			3 & 4 & 1 & 2 & 1 \\
			4 & 11 & 2 & 6 & 1 \\
			5 & 34 & 2 & 21 & 2 \\
			6 & 156 & 3 & 112 & 2 \\
			7 & 1,044 & 3 & 853 & 3 \\
			8 & 12,346 & 5 & 11,117 & 4 \\
			9 & 274,668 & 5 & 261,080 & 5 \\
			10 & 12,005,168 & 7 & 11,716,571 & 6 \\
			\specialrule{.1em}{.0em}{0em} 
		\end{tabular} 
	\end{center}
\end{table}

Similarly, the number of orthogonality spaces fulfilling (L2) is \rev{given in Table \ref{table-l2}}.

\begin{table}[H]  
	\begin{center}
		\caption{}
		\label{table-l2}
		\begin{tabular}{? c | c c | c c ?} 
			\specialrule{.1em}{.1em}{0em} 
			& \multicolumn{2}{c|}{Numbers of all} & \multicolumn{2}{c?}{Numbers of connected} \\ 
			$|X|$ & OS &  (L2)-OS  & OS &  (L2)-OS \\
			\hline
			2 & 2 & 1 & 1 & 0 \\
			3 & 4 & 2 & 2 & 0 \\
			4 & 11 & 4 & 6 & 0 \\
			5 & 34 & 8 & 21 & 0 \\
			6 & 156 & 21 & 112 & 2 \\
			7 & 1,044 & 57 & 853 & 8 \\
			8 & 12,346 & 220 & 11,117 & 70 \\
			9 & 274,668 & 1,056 & 261,080 & 490 \\
			10 & 12,005,168 & 7,301 & 11,716,571 & 4,577 \\
			\specialrule{.1em}{.0em}{0em} 
		\end{tabular} 
	\end{center}
\end{table}

\rev{Consequently, the number of linear orthogonality spaces can be seen in Table \ref{table-l12} below.}

\begin{table}[H]  
	\begin{center}
		\caption{}
		\label{table-l12}
		\begin{tabular}{? c | c c | c c ?} 
			\specialrule{.1em}{.1em}{0em} 
			& \multicolumn{2}{c|}{Numbers of all} & \multicolumn{2}{c?}{Numbers of connected} \\ 
			$|X|$ & OS &  LOS  & OS &  LOS \\
			\hline
			3 & 4 & 0 & 2 & 0 \\
			4 & 11 & 1 & 6 & 0 \\
			5 & 34 & 0 & 21 & 0 \\
			6 & 156 & 1 & 112 & 0 \\
			7 & 1,044 & 0 & 853 & 0 \\
			8 & 12,346 & 1 & 11,117 & 0 \\
			9 & 274,668 & 0 & 261,080 & 0 \\
			10 & 12,005,168 &1 & 11,716,571 & 0 \\
			\specialrule{.1em}{.0em}{0em} 
		\end{tabular} 
	\end{center}
\end{table}

	In the computational part of this study, we first obtained our orthogonality space catalogue (up to isomorphism) through \verb|nauty| which is a program written in \verb|C| language, see \cite{McKay}. Afterwards, we processed these data sets in our \verb|Python| algorithms to check whether a given orthogonality space fulfills the condition (L1) and/or (L2). \rev{For more details, we refer the readers to \cite{ameql}.}

\begin{remark}\label{remarklemma} Recall that, for any orthogonality space $(X,\perp)$ fulfilling (L1) of finite rank $m$, we know from 
\cite[Lemma 3.5]{Vet3} that ${\mathcal C}(X, \perp)$ is an atomistic, modular ortholattice of length $m$. In particular, this implies that any two maximal orthogonal subsets of $(X,\perp)$ have the same cardinality $m$.
\end{remark}

\begin{remark}
	However, the property given above is not valid when assuming (L2) alone. For instance, for an 8 element set $X$, we have an orthogonality space $(X,\perp)$ where
		\begin{align*}
		M_\perp \, = \{\{0, 4\}, \{0,7\}, \{1,5,7\} , \{1,6\} , \{2,5\} , \{3,6\}  \}
	\end{align*}
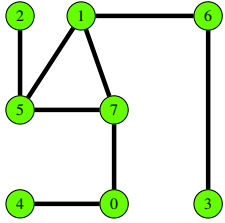
	\begin{align*}
	\begin{tikzpicture}[scale=1.25,every node/.style={draw=black,scale=0.6,circle,fill=brightgreen}]
		\node (4) at (0,0) {4};
		\node (0) at (1,0) {0};
		\node (3) at (2,0) {3};
		\node (5) at (0,1) {5};
		\node (7) at (1,1) {7};
		\node (2) at (0,2) {2};
		\node (1) at (0.65,2) {1};
		\node (6) at (2,2) {6};
		\draw[line width=0.6mm] (4) to (0);
		\draw[line width=0.6mm] (7) to (0);
		\draw[line width=0.6mm] (3) to (6);
		\draw[line width=0.6mm] (7) to (5);
		\draw[line width=0.6mm] (5) to (1);
		\draw[line width=0.6mm] (7) to (1);
		\draw[line width=0.6mm] (6) to (1);
		\draw[line width=0.6mm] (2) to (5);
	\end{tikzpicture}
\end{align*}	
that fulfills (L2). Moreover, it has rank 3 but not all maximal orthogonal subsets have the same cardinality.
\end{remark}

\begin{definition}
	An orthogonality space $(X,\perp)$ is called \rev{\em irredundant} if,
	\begin{align*}
	\{a\}^{\perp} = \{b\}^{\perp} \implies a = b \, ,
	\end{align*}
	is true for each $a,b \in X$. Moreover, $(X, \perp)$ is called \rev{\em strongly irredundant} if,
	\begin{align*}
	\{a\}^{\perp} \subseteq \{b\}^{\perp} \implies a = b \, ,
	\end{align*}
	is true for each $a,b \in X$. It is clear that, strong irredundancy implies irredundancy. 
\end{definition}

The following two properties are already captured in the proof of \cite[Lemma 5.3]{nos}. 

\begin{lemma}
	An orthogonality space fulfilling (L1) is irredundant.
\end{lemma}

\begin{proof}
	Let $a \neq b$. We have two possibilities:
	
	Suppose that $a \notperp b$. Using (L1) property, we have $c \in X$ such that $c \in \{a\}^{\perp}$ and $\{a,b\}^{\perp} = \{a,c\}^{\perp}$. Since $c \not\in  \{a,c\}^{\perp}$, we have $c \not\in  \{a,b\}^{\perp}$. However, we know that $c \in \{a\}^{\perp}$. So it is necessary to be $c \not\in \{b\}^{\perp}$. Therefore, we have $\{a\}^{\perp} \neq \{b\}^{\perp}$.
	
	Suppose that $a \perp b$. It follows $a \in \{b\}^{\perp}$. Since $a \not\in \{a\}^{\perp}$, we already have $\{a\}^{\perp} \neq \{b\}^{\perp}$.
\end{proof}

\begin{lemma}\label{str-irr}
	An orthogonality space fulfilling (L1) is strongly irredundant.
\end{lemma}

\begin{proof}
	If $\{a\}^{\perp} \subseteq \{b\}^{\perp}$, then it is enough to prove only that $\{a\}^{\perp} = \{b\}^{\perp}$ which yields $a = b$, from irredundancy. Now, suppose that $\{a\}^{\perp} \subseteq \{b\}^{\perp}$. Then, only one of the following conditions holds: 
	\begin{itemize}
		\item $\{a\}^{\perp} \subset \{b\}^{\perp}$.
		\item $\{a\}^{\perp} = \{b\}^{\perp}$.
	\end{itemize}

	We will show that, the first one is not possible. Consider that $\{a\}^{\perp} \subset \{b\}^{\perp}$. Then, we have:
	\begin{itemize}
		\item $a \neq b$ (since the orthogonality is well-defined).
		\item $\{a,b\}^{\perp} = \{a\}^{\perp}$
		\item $a \notperp b$ (because, if $a \perp b$, it follows $b \in \{a\}^{\perp} \subset \{b\}^{\perp}$ which is not possible).
	\end{itemize}

	By using (L1) property, there exists $c \in \{a\}^{\perp}$ such that $\{a,b\}^{\perp} = \{a,c\}^{\perp}$ which follows $\{a\}^{\perp} = \{a,c\}^{\perp}$. Consequently, we have $c \in \{a,c\}^{\perp\perp} = \{a\}^{\perp\perp}$ that yields the contradiction:
	\begin{itemize}
		\item $c \in \{a\}^{\perp}$ and $c \in \{a\}^{\perp\perp}$.
	\end{itemize}
\end{proof}

\begin{lemma}
	Let $(X,\perp)$ be an orthogonality space and $a \in X$. The following are equivalent: 
	\begin{enumerate}
		\item $\{a\}^{\perp} = X \setminus \{a\}$.
		\item In the graph representation of $(X,\perp)$, the node $a$ is connected to all nodes except itself.
		\item The intersection of all maximal orthogonal subsets of $(X,\perp)$ contains $a$.
	\end{enumerate} 
\end{lemma}

\begin{proof}
 \rev{This follows from the definition of orthogonality and the correspondence between orthogonality spaces and graphs given in the introduction.}
\end{proof}

\begin{proposition}
	Let $(X,\perp)$ be an orthogonality space fulfilling (L1). If we have $\{a,b\}^{\perp} = \{a\}^{\perp}$ for any $a,b \in X$, then $a = b$.
\end{proposition}

\begin{proof}
	Consider that $(X,\perp)$ is an orthogonality space fulfilling (L1). Therefore, it is strongly irredundant. Consequently, we have:
	\begin{align*}
		\{a,b\}^{\perp} = \{a\}^{\perp} & \implies \{a\}^{\perp} \subseteq \{b\}^{\perp} \\
		& \implies a = b
	\end{align*}
\end{proof}

\medskip

\begin{definition}
	An orthogonality space $(X,\perp)$ is called 
	\begin{enumerate}
	\item \rev{\em irreducible} if $X$ \rev{cannot} be partitioned into two non-empty subsets $A,B$ such 
	that $a \perp b$ for all $a \in A$ and $b \in B$; 
	\item a \rev{\em Dacey space} if, for any $A \in {\mathcal C}(X, \perp)$ and any maximal orthogonal subset $D$ of $A$, we have that $D\cc = A$.
	\end{enumerate}
\end{definition}

Note that $X$  is {irreducible} if and only if $X$ \rev{cannot} be partitioned into two non-empty orthoclosed subsets $A,B$ such 
	that $a \perp b$ for all $a \in A$ and $b \in B$.

\begin{lemma}
	Let $(X,\perp)$ be an orthogonality space  and 
	there exist an element $a \in X$ such that $\{a\}\c = X \setminus \{a\}$. Then, $(X,\perp)$ is not irreducible.
\end{lemma}

\begin{proof}
	We have a partition of $X$ being $A := \{a\}$ and \mbox{$B := X \setminus \{a\}$} which is enough 
	to say $(X,\perp)$ is not irreducible.
\end{proof}

The following result is contained in \cite[Theorem 5.6]{nos}.

\begin{lemma}
	An orthogonality space $(X,\perp)$ of finite rank is linear if and only if $X$ is an irreducible, strongly irredundant Dacey space. 
\end{lemma}

So we immediately obtain 

\begin{corollary}\label{centralnotlinear}
	Let $(X,\perp)$ be an orthogonality space of finite rank, and there exist an element $a \in X$ such that \mbox{$a\c = X \setminus \{a\}$}. Then, $(X,\perp)$ is not linear.
\end{corollary}

\begin{remark}\label{e-f}
Let $a$ be an element of an orthogonality space $(X,\perp)$ such that $\{a\}^{\perp} = X \setminus \{a\}$. 
The validity of the (L1) condition, i.e.
\begin{itemize}
	\item[(L1)] if $e \notperp f$ for distinct $e,f$; there exists $g \perp e$ such that $\{ e, f \}\c = \{ e, g \}\c$
\end{itemize}
can not be checked through $a$. In other words, none of $e$ and $f$ can be replaced with $a$ while we are checking the condition since there is no $f \in X$ such that $a \notperp f$.
\end{remark}

\noindent Moreover, we also have the following characterization of (L1) property.

\begin{lemma}
	An orthogonality space $(X, \perp)$ of finite rank fulfills (L1) if and only if $X$ is a strongly irredundant Dacey space.
\end{lemma}

\begin{proof}
	Let $(X,\perp)$ be an orthogonality space of finite rank that fulfills (L1). We know from Proposition \ref{str-irr} that it is strongly irredundant. Moreover, since $\mathcal{C}(X,\perp)$ is an orthomodular lattice, it is a Dacey space \rev{due to \cite[Lemma 3.5]{Vet3}}.
	
	Conversely, let $(X,\perp)$ be a strongly irredundant Dacey space of finite rank. Suppose that $e,f \in X$ such that $e \not\perp f$. 
	Following literally  the proof of \cite[Theorem 5.6]{nos} we can find an element $g\in X$ such 
	that  $e \perp g$ with $\{e,f\}^{\perp} = \{e,g\}^{\perp}$.
\end{proof}

\begin{proposition}
	For each finite orthogonality space with $k$ elements and rank $n$ fulfilling (L1) there exists an orthogonality space fulfilling (L1) with $k+l$ elements and rank $n+l$, for all $l \in \mathbb{N}$.   
\end{proposition}

\begin{proof}
	Let $(X,\perp)$ be an arbitrary orthogonality space with $k$ elements and rank $n$ fulfilling (L1). Fix distinct elements $x'_1, x'_2, \dots , x'_l \notin X$ for any $l \in \mathbb{N}$, and consider the set \mbox{$X' = X \cup \{x'_1, x'_2, \dots , x'_l\}$}. For each maximal orthogonal subset \mbox{$D \in \, M_\perp$}, we put $D' = D \cup \{x'_1, x'_2, \dots , x'_l\}$ and define $M_{\perp'}$ as the set of maximal orthogonal subsets. We straightforwardly have an orthogonality space \mbox{$(X',\perp')$} of rank $n+l$. Moreover, $(X',\perp')$ fulfills (L1) since any of $x'_1, x'_2, \dots , x'_l$ does not effect (L1) condition from \mbox{Remark \ref{e-f}}.
\end{proof}

\begin{proposition}
	Let $(X,\perp)$ be a finite orthogonality space with $(k+l)$ elements and rank $n+l$ fulfilling (L1), such that the intersection of all maximal orthogonal subsets has cardinality $l$. Then, there exists an orthogonality space fulfilling (L1) with $k$ elements and rank $n$.   
\end{proposition}

\begin{proof}
	\rev{By assumption, the set $Y = \bigcap \, \{D \colon D \in M_{\perp}\}$ has $l$ elements.} Define $X' = X \setminus Y$ and $D' = D \setminus Y$, \rev{for each $D \in M_{\perp}$}. We have an orthogonality space $(X',\perp')$ with $k$ elements and rank $n$. Also from Remark \ref{e-f}, it fulfills (L1).
\end{proof}

\begin{lemma}\label{difference-singleton}
	For an orthogonality space of finite rank  fulfilling (L1), the difference set of two maximal orthogonal subsets can not be singleton.
\end{lemma}

\begin{proof}
	Let $(X,\perp)$ be an orthogonality space that fulfills (L1) of finite rank. Suppose that we have two maximal orthogonal subsets $D_1 , D_2 \in \, M_\perp$ such that $D_1 \setminus D_2 = \{a\}$. We know that:
	\begin{align*}
		D_1 & = (D_1 \cap D_2) \cup (D_1 \setminus D_2) \, , \\
		D_2 & = (D_1 \cap D_2) \cup (D_2 \setminus D_1) \, .
	\end{align*}
	Since $D_1$ and $D_2$ have the same cardinality, there exist $b \in X$ such that $D_2 \setminus D_1 = \{b\}$. 
	If $a \perp b$, then $a \in D_2 ^\perp$ which is a contradiction. So we must have $a \notperp b$.
	
	Now, by using (L1) condition, there exist $c \in \{a\}^{\perp}$ such that $\{ a, b \}\c = \{ a, c \}\c$. Then, we have $D_1 \cap D_2 \subseteq \{ a, b \}\c = \{ a, c \}\c$ that yields $c \in D_1\c$. This yields a maximal orthogonal subset $D_1 \cup \{c\}$ which is a contradiction.
\end{proof}


\section{Rank and distance in  orthogonality spaces fulfilling (L1)}

In this section we investigate graph-theoretic structure of orthogonality spaces fulfilling (L1). In particular, we describe 
how linear orthogonality spaces of rank 2 look like up to isomorphism and we point out that they are not connected as graphs.  
Moreover, finite linear orthogonality spaces  of rank 2  have always even cardinality of at least four.
In contrast, we prove that orthogonality spaces of rank at least 3  fulfilling (L1) are always connected and that 
there are no finite  linear orthogonality spaces of rank at least 3.

For later considerations, we introduce a further, particularly simple example which describes in fact 
a balanced biregular bipartite graph of degree one which is not connected.

\begin{example}\label{exrank2L1}
Let $A$ and $B$ be sets such that $|A|=|B|$ and $A\cap B=\emptyset$. Then there is a bijection $\varphi\colon A\to B$. 
Let us denote by ${\mathbf 2}(A,B,\varphi)$ the orthogonality space: 
$$(A\cup B, \{(a, \varphi(a)) \mid a\in A\}\cup \{(\varphi(a),a) \mid a\in A\}).$$ Evidently, ${\mathbf 2}(A,B,\varphi)$ has rank 2 and any maximal subset of mutually orthogonal elements 
of $A\cup B$ has exactly 2 elements.
\end{example}

From now on, we fix the notation ${\mathbf 2}(A,B,\varphi)$ for the orthogonality space defined above.

\begin{proposition}\label{rank2L1}
Let $(X,\perp)$ be an orthogonality space of rank 2 fulfilling (L1). Then $X={\mathbf 2}(A,B,\varphi)$ for some subsets 
$A, B\subseteq X$ and a bijection $\varphi\colon A\to B$. Moreover, any  orthogonality space  $(X,\perp)$ 
of the above form has rank 2 and fulfills (L1). 
\end{proposition}
\begin{proof} First, we check that $|\{a\}^{\bot}|=1$ for all $a\in X$. Let $a\in X$. Assume first that 
$|\{a\}^{\bot}|=0$. Since $|X|\geq 2$ there is \rev{a} $c\in X$ such that $c\not\perp a$. From (L1) we obtain that there is \rev{a} $d\in X$ such that 
$\{a,c\}^{\bot}=\{a,d\}^{\bot}$ and $a\perp d$, a contradiction with $|\{a\}^{\bot}|=0$. 
Assume now that $|\{a\}^{\bot}|\geq 2$. Then there are $b, c\in X$ such that $a\perp b$ and $a\perp c$. 
Since $X$ has  rank 2 we obtain that $b\not\perp c$.  From (L1) we obtain that there \rev{exists} $d\in X$ such that  
$a\in \{b,c\}^{\bot}=\{b,d\}^{\bot}$ and $b\perp d$. Hence $\{a, b, d\}$ is an orthogonal set, i.e., 
$X$ has  rank at least 3, a contradiction. 

It follows that there are subsets $A, B\subseteq X$,  $A\cap B=\emptyset$ and a bijection $\varphi\colon A\to B$ 
such that \rev{$x \perp y$} if and only if $x\in A$ and $y=\varphi(x)$ or  $y\in A$ and $x=\varphi(y)$. 

Hence $(X,\perp)={\mathbf 2}(A,B,\varphi)$. 

Now, let  $(X,\perp)={\mathbf 2}(A,B,\varphi)$ for some 
disjoint sets $A$ and $B$, and a bijection $\varphi\colon A\to B$.
Assume first that $|A|=1$. Then any two elements of $X$ are orthogonal, i.e.,  $(X,\perp)$ fulfills (L1) and it has rank 2. 
Assume now $|A|\geq 2$ and let $a\not\perp b$. Then $ \{a,b\}^{\bot}=\emptyset$ and there exists 
$c\in X$ such that $a\perp c$ and $ \{a,c\}^{\bot}=\emptyset$. Again,  $(X,\perp)$ fulfills (L1) and it has \mbox{rank 2}. 
\end{proof}

\begin{remark} Let $n$ be a natural number. Recall that $\text{MO}(n)$ is the horizontal sum of $n$ copies of the four 
element Boolean algebra (see \cite{Greechie}). Let $(X,\perp)$ be an orthogonality space 
such that $\mathcal{C}(X,\perp)\cong \text{MO}(n)$ as an  ortholattice. Then $(X,\perp)$ has 
rank 2,  fulfills (L1) and  $|X|=2n$. 
\end{remark}

\begin{corollary}\label{rank2numberl1} Let $X$ be a finite set such that $|X|=2n$ for some natural number $n$. 
Then the number of all orthogonality spaces on $X$ of rank 2  and  fulfilling (L1) is 
$$(2n-1) \cdot (2n-3) \,\, {\cdots} \,\, 3 \cdot 1=\frac{(2n)!}{n! 2^{n}}$$
\end{corollary}
\begin{proof}
Let $|X|=2n$. Due to Proposition \ref{rank2L1}, any orthogonality space $(X,\perp)$ 
that fulfills (L1) of rank 2 must be a partition of $X$ to $n$.

In other words, $\perp$ has $n$ maximal orthogonal subsets in which any of them has 2 
elements and they are mutually disjoint. For any $a_1 \in X$, we can write $2n-1$ different $a'_1 \in X$ such that $a_1 \perp a ' _1$ ($a_1 \neq a'_1$ because of the antisymmetry). It follows that, for any  $a_2 \in X \setminus \{a_1,a'_1\}$ we can write $2n-3$ different possible $a'_2 \in X$ such that $a_2 \perp a' _2$. By iteration, at the end, for $a_{n} \in X \setminus \{a_1,a'_1, a_2,a'_2 \cdots, a_{n-1},a'_{n-1}\}$, we have only one possible $a'_n$ such that $a_n \perp a'_n$. So 
there are exactly  $(2n-1)\cdot(2n-3)  \,\, {\cdots} \,\,  3\cdot1$ 
different orthogonality spaces $(X,\perp)$.

\end{proof}

\begin{remark}
	Our \verb|Python| computations reveal that, for any orthogonality space $(X,\perp)$ 
	of rank $3$ that fulfills (L1) where $4\leq|X|\leq 10$, the following must hold:
\begin{itemize}
	\item $|X|$ is odd.
	\item $M_{\perp}$ is uniquely determined up to isomorphism. 
	\item We always have exactly one (fixed) element $a \in X$ such that $a^{\perp} = X \setminus \{a\}$.
\end{itemize}	
So we observe that, we only have the following four orthogonality spaces of rank $3$ that fulfill (L1) (for $|X|\leq 10$):

\medskip

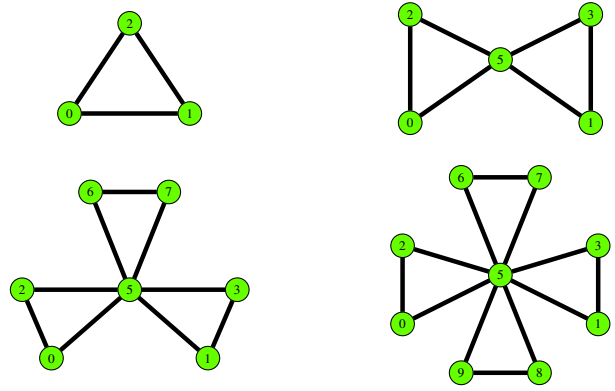
\begin{tcolorbox}[sidebyside, blanker, halign=center, halign lower=center]
	\begin{tikzpicture}[scale=0.8,every node/.style={draw=black,scale=0.5,circle,fill=brightgreen}]
		\node (0) at (0,0) {0};
		\node (1) at (2,0) {1};
		\node (2) at (1,1.5) {2};
		\draw[line width=0.6mm] (0) to (1);
		\draw[line width=0.6mm] (0) to (2);
		\draw[line width=0.6mm] (2) to (1);
	\end{tikzpicture} 

	\tcblower

	\begin{tikzpicture}[scale=1.2,every node/.style={draw=black,scale=0.5,circle,fill=brightgreen}]
		\node (0) at (0,0.3) {0};
		\node (2) at (0,1.5) {2};
		\node (5) at (1,1) {5};
		\node (1) at (2,0.3) {1};
		\node (3) at (2,1.5) {3};
		\draw[line width=0.6mm] (0) to (2);
		\draw[line width=0.6mm] (0) to (5);
		\draw[line width=0.6mm] (2) to (5);
		\draw[line width=0.6mm] (1) to (3);
		\draw[line width=0.6mm] (1) to (5);
		\draw[line width=0.6mm] (3) to (5);
	\end{tikzpicture}
\end{tcolorbox}

\medskip

\begin{tcolorbox}[sidebyside, blanker, halign=center, halign lower=center]
	
	\begin{tikzpicture}[scale=1.3,every node/.style={draw=black,scale=0.5,circle,fill=brightgreen}]
		\node (0) at (0.2,0.3) {0};
		\node (2) at (-0.1,1) {2};
		\node (5) at (1,1) {5};
		\node (1) at (1.8,0.3) {1};
		\node (3) at (2.1,1) {3};
		\node (6) at (0.6,2) {6};
		\node (7) at (1.4,2) {7};
		\draw[line width=0.6mm] (0) to (2);
		\draw[line width=0.6mm] (0) to (5);
		\draw[line width=0.6mm] (2) to (5);
		\draw[line width=0.6mm] (1) to (3);
		\draw[line width=0.6mm] (1) to (5);
		\draw[line width=0.6mm] (3) to (5);
		\draw[line width=0.6mm] (6) to (7);
		\draw[line width=0.6mm] (6) to (5);
		\draw[line width=0.6mm] (7) to (5);
	\end{tikzpicture}
	
	\tcblower
	
	\begin{tikzpicture}[scale=1.3,every node/.style={draw=black,scale=0.5,circle,fill=brightgreen}]
	\node (0) at (0,0.5) {0};
	\node (2) at (0,1.3) {2};
	\node (5) at (1,1) {5};
	\node (1) at (2,0.5) {1};
	\node (3) at (2,1.3) {3};
	\node (6) at (0.6,2) {6};
	\node (7) at (1.4,2) {7};
	\node (8) at (1.4,0) {8};
	\node (9) at (0.6,0) {9};
	\draw[line width=0.6mm] (0) to (2);
	\draw[line width=0.6mm] (0) to (5);
	\draw[line width=0.6mm] (2) to (5);
	\draw[line width=0.6mm] (1) to (3);
	\draw[line width=0.6mm] (1) to (5);
	\draw[line width=0.6mm] (3) to (5);
	\draw[line width=0.6mm] (6) to (7);
	\draw[line width=0.6mm] (6) to (5);
	\draw[line width=0.6mm] (7) to (5);
	\draw[line width=0.6mm] (8) to (9);
	\draw[line width=0.6mm] (8) to (5);
	\draw[line width=0.6mm] (9) to (5);
	\end{tikzpicture}
\end{tcolorbox}

\end{remark}

\begin{proposition}\label{rank3finite} 
	Let $A,B$ be two disjoint sets with a bijection $\varphi \colon A \to B$, and \rev{let $x$ be an element not contained in $A$ or $B$.} We denote the orthogonality space $(X,\perp) = {\mathbf 3}(A,B,\varphi)$ in which the maximal orthogonal subsets are $\{x,a,\varphi(a)\}$ for all $a \in A$. Then ${\mathbf 3}(A,B,\varphi)$ fulfills (L1), is not linear and has rank 3.
\end{proposition}
\begin{proof}
	Let $a \not\perp b$. It means $a,b$ are both different than $x$. Then we have $a \perp \varphi (a)$ such that $x = \{a,b\}^{\perp} = \{a,\varphi(a)\}^{\perp}$. From Corollary \ref{centralnotlinear} and the fact that 
	${x}^{\perp}=X\setminus \{x\}$ we obtain that ${\mathbf 3}(A,B,\varphi)$  is not linear.
\end{proof}

\begin{theorem}\label{open2}
	Any finite orthogonality space $(X,\perp)$ of rank 3 
	fulfilling (L1) is \rev{of the form} ${\mathbf 3}(A,B,\varphi)$.
\end{theorem}
\begin{proof} If $|X|=3$ then $M_{\bot}=\{X\}$ and the statement is valid. Assume now that $|X|>3$. 
Suppose first that $X$ is not irreducible. Then there are disjoint non-empty orthoclosed subsets $U$ and $V$ such that 
$U\cup V=X$ and $U=V^{\bot}$ and $V=U^{\bot}$. We can assume that $U=\{x\}$ for a suitable element $x\in X$ and 
$V=\{x\}^{\bot}$ (otherwise we interchange $U$ with $V$). Since $V$ is an orthogonality space of rank 2 fulfilling (L1) 
there are subsets $A, B\subseteq V$ and a bijection $\varphi\colon A\to B$ such that $V={\mathbf 2}(A,B,\varphi)$. 
Since $V=\{x\}^{\bot}$ we obtain that $X={\mathbf 3}(A,B,\varphi)$.

Suppose now that $X$ is irreducible. Since $|\mathcal{C}(X,\perp)|\geq 4$, \rev{it cannot be isomorphic to a point or two element chain. Consequently,} following from \cite[Theorem 16]{Greechie} and \mbox{Remark \ref{remarklemma}} \rev{we conclude that}  $\mathcal{C}(X,\perp)\cong \text{MO}(n)$. Hence $(X,\perp)$ is of \mbox{rank 2}, a contradiction.
\end{proof}

\begin{proposition}\label{rank2Linear}
Let $(X,\perp)$ be a linear orthogonality space of rank 2. Then $(X,\perp)={\mathbf 2}(A,B,\varphi)$ for some 
disjoint sets $A$ and $B$,  $|A|\geq 2$ and a bijection $\varphi\colon A\to B$. 
Moreover, any  orthogonality space  $(X,\perp)$ 
of the above form has rank 2 and is linear. 
\end{proposition}

\begin{proof} From Proposition \ref{rank2L1} we know that 
$(X,\perp)={\mathbf 2}(A,B,\varphi)$ for some 
disjoint sets $A$ and $B$,   and a bijection $\varphi\colon A\to B$. 

 Assume that $|A|=1$. We know   
$a\perp \varphi(a)=b$. Since any two different elements of $X$ are orthogonal there is no element 
$c\in X\setminus \{a, b\}$ such that 
$\{a,b\}^{\bot}=\{a,c\}^{\bot}$ and $a\not\perp c$. 

Conversely, let  $(X,\perp)={\mathbf 2}(A,B,\varphi)$ for some 
disjoint sets $A$ and $B$,  $|A|\geq 2$ and a bijection $\varphi\colon A\to B$. By 
 Proposition \ref{rank2L1} it remains to show that 
$(X,\perp)$ fulfills (L2). Let $a\perp b$. Then there exists $c\in X\setminus \{a, b\}$. Clearly, $a\not\perp c$ and 
$ \{a,b\}^{\bot}=\emptyset= \{a,c\}^{\bot}$.
\end{proof}

\begin{corollary}\label{rank2Linearnotconn}
Any linear orthogonality space of rank 2 is not connected. 
\end{corollary}


\begin{remark} Recall that Corollary \ref{rank2numberl1} \rev{for linear orthogonality spaces} is related to the results of 
Eckmann and  Zabey \cite{EZ}. If we look e.g. on a finite set $X$ with $|X|=16$ then there is no finite field $F$ and 
a vector space $V$ of dimension $2$ such that the lattice of all subspaces of $V$ is isomorphic to 
${\mathcal C}(X, \perp)$ for some linear orthogonality space on $X$ of rank 2. The reason is that this isomorphism 
would require $16=p^{d}+1$ for some prime number $p$ which is impossible.
\end{remark}

\begin{proposition}\label{rank3L1}
Let $(X,\perp)$ be an orthogonality space of rank at least 3 fulfilling (L1). Then $X$ is connected and its diameter is at most 2. 
Moreover, 
\begin{enumerate}[{\rm(D3)}]
\item $a\not\perp b$  if and only if $d(a,b)=2$.
\end{enumerate}
\end{proposition}
\begin{proof}
 Assume that $a\not\perp b$. 
From (L1) we obtain that there is $c\in X$ such that $a\perp c$ and $ \{a,b\}^{\bot}= \{a,c\}^{\bot}$. 
Since $(X,\perp)$ has rank at least 3 and fulfills (L1) there is an element $e\in X$ such that $e\in  \{a,c\}^{\bot}$. 
Hence $e\in  \{a,b\}^{\bot}$ and $d(a,b)=2$. From (D1) we obtain the remaining implication of (D3). The 
statement then follows from the fact that the distance of any two elements of $X$ is at most 2. 
\end{proof}

\begin{definition}
By an {\it sfield}, we mean a skew field (i.e., a division ring). Let $V$ be a linear space over an sfield $K$. \rev{In accordance with Example \ref{ex:standard-example-1}, we define \mbox{$P(V) = \{ \lin x \colon x \in V\setminus \{0\} \}$}} to be the projective space associated with $V$. 

A {\it $\star$-sfield} is an sfield equipped with an involutorial antiautomorphism $^\star$. An {\it (anisotropic) Hermitian space} is a linear space $H$ over a $\star$-sfield $K$ that is equipped with an anisotropic, symmetric sesquilinear form $\herm{\cdot}{\cdot} \colon H \times H \to K$. 
\end{definition}

The following correspondence between linear orthogonality spaces and linear spaces was shown in \cite{Vet3}.

\begin{theorem} \label{thm:orthogonality-spaces-by-orthomodular-spaces}
Let $H$ be a Hermitian space of finite dimension $n$. Then $(P(H), \perp)$ is a linear orthogonality space of rank $n$.

Conversely, let $(X, \perp)$ be a linear orthogonality space of finite rank $n \geq 4$. Then there is a $\star$-sfield $K$ and an $n$-dimensional Hermitian space $H$ over $K$ such that $(X,\perp)$ is isomorphic to $(P(H), \perp)$.
\end{theorem}

We can summarize the preceding results in this section as follows.

\begin{theorem} \label{thm:summar}
Let $(X, \perp)$ be a linear orthogonality space of finite rank $m$.  Then 
\begin{enumerate}[{\rm(i)}]
\item If $m=2$ then $X$ is either finite with even cardinality or infinite.
\item If $m\geq3$ then $X$ is infinite.
\end{enumerate}
\end{theorem}
\begin{proof} (i) It follows from Proposition \ref{rank2Linear}.

(ii) If $m=3$ the statement follows from Theorem \ref{open2} and Proposition \ref{rank3finite}. Let $m\geq 4$. 
From Theorem \ref{thm:orthogonality-spaces-by-orthomodular-spaces} we know that
 $(X,\perp)$ is isomorphic to $(P(H), \perp)$ for some $\star$-sfield $K$ and an $n$-dimensional 
 Hermitian space $H$ over $K$. Since $H$ has dimension at least 4, we have by the results of  Eckmann and  Zabey \cite{EZ} that 
 $K$ is infinite. Hence also $(P(H), \perp)$ is infinite.
\end{proof}

\section*{Acknowledgment}

Research of the first and fourth author was supported {
by the Austrian Science Fund (FWF), project I~4579-N, and the Czech Science Foundation (GA\v CR), project 20-09869L, 
entitled ``The many facets of orthomodularity''.}
Research of the second author was supported by the project  ``New approaches to aggregation operators in analysis and processing
of data'', Nr.~18-06915S by Czech Grant Agency (GA\v{C}R).
Research of the third author was supported by the project ``Group Techniques and Quantum Information'', 
No.~MUNI/G/1211/2017 by Masaryk University Grant Agency (GAMU). 



\end{document}